\documentclass[12pt,a4paper]{article}
\usepackage{t1enc}
\usepackage[utf8]{inputenc}
\usepackage[english]{babel}
\usepackage{amsmath}
\usepackage{amssymb}
\usepackage{amsthm}
\usepackage{hyperref}

\newtheorem{theorem}{Theorem}[section]
\newtheorem{lemma}[theorem]{Lemma}
\newtheorem{cor}[theorem]{Corollary}

\newtheorem{prop}[theorem]{Proposition}

\newtheorem{exa}[theorem]{Example}

\def\cC{\mathcal C}

\def\cF{\mathcal F}

\def\cH{\mathcal H}

\def\cL{\mathcal L}
\def\cM{\mathcal M}

\def\cX{\mathcal X}

\def\ord{\mbox{\rm ord}}
\def\deg{\mbox{\rm deg}}

\def\div{\mbox{\rm div}}
\def\min{{\rm min}}

\def\Im{\mbox{\rm Im}}

\def\Div{\mbox{\rm div}}

\def\fq{{\mathbb F}_q}

\def\bfq{{\overline{\mathbb F}}_q}

\def\gg{\mathfrak{g}}
\def\F{\mathbb F}

\title{Number of rational branches of a singular plane curve over a finite field}

\author{Nazar Arakelian}

\begin{document}

\maketitle
\begin{abstract}
Let $\cF$ be a plane singular curve defined over a finite field $\F_q$. Via results of \cite{SV} and \cite{AB}, the linear system of plane curves of a given degree passing through the singularities of $\cF$ provides potentially good bounds for the number of points on a non-singular model of $\cF$. In this note, the case of a curve with two singularities such that the sum of their multiplicities is precisely the degree of the curve is investigated in more depth. In particular, such plane models are completely characterized, and for $p > 3$, a curve of this type attaining one of the obtained bounds is presented.
\end{abstract}


\section{Introduction}\label{intro}

For a prime number $p$, let $\F_q$ be the finite field with $q$ elements, where $q$ is a power of $p$. Let $\cF$ be a (projective, geometrically irreducible, algebraic) plane curve of degree $d$, genus $\gg$, defined over $\F_q$. Denote by $N_m(\cF)$ the number of $\F_{q^m}$-rational points on a non-singular model of $\cF$ (or equivalently, the number of $\F_{q^m}$-rational branches of $\cF$), where $m \geq 1$ is an integer. One of the most challenging problems regarding $\cF$ is the determination of the number $N_m(\cF)$ in terms of $q,d,\gg$ and other covariants. In fact, there are only few families of curves for which such an explicit formula for $N_m(\cF)$ is known, see \cite{Bo2,CV,GSX,HV} for instance. This fact has been motivating the search for estimates for $N_m(\cF)$ over the last years. The most remarkable of them is the Hasse-Weil bound, presented by A. Weil in the 1940s, which states   
\begin{equation}\label{hw}
|N_m(\cF)-(q^m+1)| \leq 2\gg\sqrt{q^m}.
\end{equation}

Introduced in 1986, the St\"ohr-Voloch technique is also one of the most significant approaches to estimate $N_m(\cF)$ \cite{SV}. This technique depends on the linear series of the curve. More precisely, let $g_n^r$ be a base-point-free linear series on $\cF$ of degree $n$ and dimension $r$, defined over $\F_{q^m}$. The St\"ohr-Voloch main theorem \cite[Theorem 2.13]{SV} states that, associated to $g_n^r$ and $q^m$, there exists a sequence of non-negative integers $(\nu_0,\ldots,\nu_{r-1})$, with $0=\nu_0<\cdots<\nu_{r-1}$, such that
\begin{equation}\label{sv}
N_m(\cF) \leq \displaystyle\frac{(\nu_1+\cdots+\nu_{r-1})(2\gg-2)+(q^m+r)n}{r}.
\end{equation}

The bound (\ref{sv}) improves the Hasse-Weil bound in many instances (see \cite{SV}). In \cite{AB} a bound obtained via a variation of St\"ohr-Voloch method is presented: if $g_n^r$ is defined over $\fq$, given positive integers $u,m$ with $u<m$ and $\gcd(u,m)=1$, there are positive integers $c_i$  (depending on $i$ and $g_n^r$) with $i=1,u,m,m-u$,  such that
\begin{eqnarray}\label{ab}
& & (c_1-c_u-c_m-c_{m-u})N_1(\cF)+c_uN_u(\cF)+c_mN_m(\cF)+c_{m-u}N_{m-u}(\cF)  \\ \nonumber
  &\leq &(\kappa_0+\cdots+\kappa_{r-2})(2\gg-2)+(q^u+q^m+r-1)n,
\end{eqnarray}
where $(\kappa_0,\ldots,\kappa_{r-2})$ is a sequence of non-negative integers, which depends also on $g_n^r$, such that  $0 \leq \kappa_0<\cdots<\kappa_{r-2}$, see \cite[Theorem 4.4]{AB}. Here, we have $c_1 \geq q^u+2(r-1)$ and $c_{m-u} \geq q^u$, and these numbers can be bigger depending on some properties of $\cF$ and $g_n^r$. As shown in \cite{AB}, the bound (\ref{ab}) improves Hasse-Weil and St\"ohr-Voloch bounds in many situations.

It can be seen directly from (\ref{sv}) that the application of the St\"ohr-Voloch method to a linear series $g_n^r$ such that $r$ is large compared to $n$, has good chances to yield efficient upper bounds for the number of rational points on a non-singular model of the curve. The same holds for bound (\ref{ab}). In this sense, given a plane singular curve $\cF$, we show that the bounds for the number of its rational points arising from linear systems of plane curves of the same degree passing through the singularities of $\cF$ are potentially good (Section \ref{ls}). The case in that $\cF$ satisfies the following property is investigated more in depth (Section \ref{twosing}):
\begin{itemize}
\item [(H)] $\cF$ has at least two singularities $P_1$ and $P_2$ defined over $\F_q$ with multiplicities $r_1$ and $r_2$, respectively, such that $r_1+r_2=d$.
\end{itemize}
 In particular, in Theorem \ref{cotaSVH}, we present the bounds (\ref{sv}) and (\ref{ab}) obtained via the linear system of conics passing through the sigularities of a curve $\cF$ satisfying (H). For suitable values of $d$, these bounds are better than bounds arising from complete linear systems of lines and conics.

In Section \ref{att}, the plane models satisfying such property are completely characterized via their equations and the respective generators of the function fields (Theorem \ref{carac}). Moreover, in Proposition \ref{toda}, it is shown that every projective algebraic curve with at least two rational branches has a plane model for which (H) holds. We finish Section \ref{att} by presenting, for every $q$ with $p>3$, a plane curve $\cM$ satisfying (H), defined over $\F_q$, that attains the bound (\ref{ab}) obtained in this way.

\section{Background and notation}\label{bg}

Notation and terminology are standard. Our main references are \cite{AB,HKT,SV}. As before, $\mathbb{F}_q$ denotes a finite field of order $q$, where $q$ is a power of a prime number $p$. Let $\bfq$ denote the algebraic closure of $\mathbb{F}_q$. Furthermore, $\cF$ stands for an algebraic projective geometrically irreducible plane curve of degree $d$, genus $\gg$ and with affine equation 
\begin{equation}
\label{eqfxy} f(X,Y)=0,
\end{equation}
defined over $\mathbb{F}_q$. Denote by $\cX$ a non-singular model of $\cF$. Then the points on $\cX$ are in correspondence with the branches of $\cF$, see \cite[Theorem 5.29]{HKT}. The function field of $\cF$ is $\mathbb{F}_q(\cF)=\mathbb{F}_q(\cX)=\mathbb{F}_q(x,y)$ with $f(x,y)=0$.

Let $\cH$ be a plane curve such that $\cF$ is not a component of $\cH$. For a branch $\gamma$ of $\cF$ centered at $P^{\prime} \in \cF$, the intersection multiplicity between $\gamma$ and $\cH$ is denoted by $I(P^{\prime},\cH \cap \gamma)$. If $P \in \cX$ is the point associated to $\gamma$, and $\cH$ is defined by the equation $h(X,Y)=0$, where $h(X,Y) \in \F_q[X,Y]$, then $I(P^{\prime},\cH \cap \gamma)=v_P(h(x,y))$, where $v_P$ denotes the discrete valuation with respect to $P$ (\cite[Definition 4.34]{HKT}). Moreover, according to \cite[Theorem 4.36]{HKT}, the intersection multiplicity between $\cF$ and $\cH$ at some $Q \in \mathbb{P}^{2}(\bfq)$, denoted by $I(Q,\cF \cap \cH)$, satisfies
\begin{equation}\label{muint}
I(Q,\cF \cap \cH)= \displaystyle\sum_{\gamma}I(Q,\cH \cap \gamma),
\end{equation}
with the sum running over all branches $\gamma$ of $\cF$ centered at $Q$. Let $Q \in \mathbb{P}^2(\bfq)$ and denote by $m_Q(\cF)$ the multiplicity of $Q$ in $\cF$. A line $\ell$ is said to be tangent to $\cF$ at $Q$ if $I(Q,\cF \cap \ell)>m_Q(\cF)$.

The intersection divisor of $\cH$ is defined by
$$
\cH \cdot \cF= \displaystyle\sum_{P \in \cX} I(P^{\prime},\cH \cap \gamma_P)P,
$$
where $\gamma_P$ is the branch of $\cF$ corresponding to $P \in \cX$ centered at $P^{\prime} \in \cF$.

Let $\Sigma$ be a linear system of plane curves of degree $t$. Then $\Sigma$ cuts out on $\cF$ a (not necessarily complete) linear series $g_{\tilde{n}}^r$, where $\tilde{n}$ and $r$ denote the degree and the dimension of $g_{\tilde{n}}^r$ respectively, \cite[Theorem 6.46]{HKT}. The divisors in $g_{\tilde{n}}^r$ are precisely the intersection divisors $\cH \cdot \cF$, with $\cH \in \Sigma$. Thus $\tilde{n}=td$ by B\'ezout's Theorem. Let $B$ be an effective divisor such that $D \geq B$ for all $D \in g_{\tilde{n}}^r$. The divisor $B$ is called the base locus of $g_{\tilde{n}}^r$. A point $P$ in the support of $B$ is called a base point of $g_{\tilde{n}}^r$. If $B=0$, then the linear series $g_{\tilde{n}}^r$ is said to be base-point-free. Let $E$ be an effective divisor on $\cF$ such that $g_{\tilde{n}}^r \subseteq |E|$ (here, |E| denotes the complete linear series determined by $E$). Note that, for $f \in \bfq(\cF)$, if $\div(f)+E \geq B$ then $\div(f)+(E-B) \geq 0$. Moreover, a linearly independent set in $\cL(E)$ remains linearly independent in $\cL(E-B)$, where $\cL(R)$ denotes the Riemann-Roch space of a divisor $R$. Thus $\Sigma$ cuts out on $\cF$ a base-point-free linear series $g_{{\tilde{n}}-\deg B}^r \subseteq |E-B|$. Roughly speaking, we exclude the base locus of $g_{\tilde{n}}^r$, and then we obtain a base-point-free linear series $g_{n}^r$, where $n={\tilde{n}}-\deg B$.

The base-point-free linear series $g_{n}^r$ gives rise to a model of $\cF$ of degree $n$ in $\mathbb{P}^r(\bfq)$, see e.g. \cite[Section 7.4]{HKT}. From \cite[Section 1]{SV}, given a branch $\gamma$ of $\cF$, there exists a sequence of non-negative integers $(j_0(\gamma), \ldots, j_r(\gamma))$, such that $j_0(\gamma)< \ldots<j_r(\gamma)$, called the order sequence of $\gamma$ with respect to $g_n^r$, which is defined by all the possible intersection multiplicities of $\gamma$ with some curve of $\Sigma$. For a non-singular point $P \in \cF$, this sequence is also called order sequence of $P$, and its terms denoted by $j_i(P)$, for $i=0,\ldots,r$. From \cite[Section 1]{SV}, except for a finite number of branches of $\cF$, this sequence is the same. Such sequence is called the order sequence of $\cF$ with respect to $\Sigma$ (or $g_n^r$), and it is denoted by $(\varepsilon_0, \ldots,\varepsilon_r)$.   When $(\varepsilon_0,\ldots,\varepsilon_r)=(0,\ldots,r)$, the curve $\cF$ is said to be classical with respect to $\Sigma$; otherwise, $\cF$ is called non-classical.

For a given branch $\gamma$ of $\cF$ centered at $Q$, by \cite[Theorem 1.1]{SV}, there exists a unique curve $\cC_\gamma \in \Sigma$ such that
$$
I(Q,\cC_\gamma \cap \gamma)=j_r(\gamma).
$$
The curve $\cC_\gamma$ is the osculating curve (with respect to $\Sigma$) to $\cF$ at $\gamma$. When $\Sigma$ is the linear system of lines, then $\cC_\gamma$ is called the tangent line at $\gamma$. In this case, one can show that $\cC_\gamma$ is tangent to $\cF$ at $Q$ in the usual sense. If $P \in \cX$ is associated to $\gamma$, we also say that $\cC_\gamma$ is the osculating curve at $P$.

Suppose that $g_n^r$ is defined over $\F_q$. Let $\{f_0,\ldots,f_r\}$ be a basis of the Riemann-Roch space associated to $g_n^r$, with $f_i \in \fq(\cF)$, and let $\tau$ be a separating variable of the extension $\fq(\cF)/\fq$. By \cite[Proposition 2.1]{SV}, there exists a sequence of non-negative integers $(\nu_0,\ldots,\nu_{r-1})$, chosen minimally in the lexicographic order, such that
\begin{equation}\label{fncl}
\left|
  \begin{array}{ccc}
  f_0^q & \ldots & f_r^q \\
  D_\tau^{(\nu_0)}f_0 & \ldots & D_\tau^{(\nu_0)}f_r \\
   \vdots & \cdots & \vdots \\
  D_\tau^{(\nu_{r-1})}f_0 & \cdots & D_\tau^{(\nu_{r-1})}f_r
  \end{array}
  \right| \neq 0,
  \end{equation}
where $D^{(k)}_\tau g$ is the $k$-th Hasse derivative of the function $g$ with respect to $\tau$. By \cite[Proposition 1.4]{SV}, this sequence does not depend on the choice of the basis $\{f_0,\ldots,f_r\}$. It is called the $\fq$-Frobenius order sequence of $\cF$ with respect to $\Sigma$ (or $g_n^r$). Again, the curve $\cF$ is $\fq$-Frobenius classical with respect to $\Sigma$ if $(\nu_0,\ldots,\nu_{r-1})=(0,\ldots, r-1)$, and $\fq$-Frobenius non-classical otherwise. Note that, once the sequence $(\nu_0,\ldots,\nu_{r-1})$ is known, bound (\ref{sv}) is determined.

Now let $u,m$ be  positive integers with $u<m$ and $\gcd(u,m)=1$. In \cite[Proposition 3.1]{AB}, it is shown that there exists a sequence of non-negative integers $(\kappa_0,\ldots,\kappa_{r-2})$, chosen minimally in the lexicographic order, such that
\begin{equation}\label{dfncl}
\left|
  \begin{array}{ccc}
  f_0^{q^m} & \ldots & f_r^{q^m} \\
	f_0^{q^u} & \ldots & f_r^{q^u} \\
  D_\tau^{(\kappa_0)}f_0 & \ldots & D_\tau^{(\kappa_0)}f_r \\
   \vdots & \cdots & \vdots \\
  D_\tau^{(\kappa_{r-2})}f_0 & \cdots & D_\tau^{(\kappa_{r-2})}f_r
  \end{array}
  \right| \neq 0.
  \end{equation}
Such sequence, which is not dependent on the basis $\{f_0,\ldots,f_r\}$ (by \cite[Proposition 3.2]{AB}), is called $(q^u,q^m)$-Frobenius order sequence of $\cF$ with respect to $\Sigma$. In this case, we also have the concept of $(q^u,q^m)$-Frobenius (non)classicality in the sense of the previous cases. In its turn, the sequence $(\kappa_0,\ldots,\kappa_{r-2})$ is the key ingredient for the determination of bound (\ref{ab}).

\section{Linear systems through singular points}\label{ls}

Let $\cF$ be a plane curve of degree $d$ defined over $\fq$ with the same notation as in Section \ref{bg}. Consider $P_1,\ldots,P_k \in \cF$ with multiplicities $r_1,\ldots,r_k$, respectively. From the discussion after \cite[Definition 5.58]{HKT} 
\begin{equation}\label{eq4} 
\gg\leq \frac{1}{2} (d-1)(d-2) - \frac{1}{2} \displaystyle\sum_{i=1}^k  r_i(r_i-1),
\end{equation}
and equality holds if $\{P_1, \ldots, P_k\}$ contains all singularities of $\cF$ and each of them is ordinary. We fix an integer $t$ with $t<d$, and positive integers $s_1,\ldots,s_k$. Let $\Sigma_t$ be the linear system of all projective (possibly reducible) plane curves $\cC$ of degree $t$ passing through each $P_i$ with multiplicity at least $s_i$. From the proof of \cite[Lemma 3.24]{HKT}, this requirement imposes at most $\textstyle\frac{1}{2} \sum_{i=1}^{k} s_i(s_i+1)$ linear conditions on $\cC$. Therefore, the integers $t,s_1,\ldots,s_k$ must be chosen in such a way that the number
\begin{equation}
\label{eq3} h=\frac{1}{2}\,t(t+3)-\displaystyle\frac{1}{2} \sum_{i=1}^k s_i(s_i+1) 
\end{equation}
is non-negative.

As explained in Section \ref{bg}, $\Sigma_t$ cuts out on $\cF$ a (not necessarily complete) linear series $g_{td}^r$ of degree $td$ and dimension $r$. A basis of the Riemann-Roch space associated to $g_{td}^r$ can be obtained via a basis of $\Sigma_t$ modulo $F(X,Y,Z)$, where $F(X,Y,Z)$ denotes the homogenization of $f(X,Y)$ with respect to $Z$. Since $t<d$, linear independence means linear independence modulo $F(X,Y,Z)$. Therefore, $r$ coincides with the dimension of $\Sigma_t$; in particular, $r\geq h$ (see e.g. \cite[Section 5.2]{Fu}). Thus, if $B$ denotes the base locus of $g_{td}^r$, we conclude that $\Sigma_t$ cuts out on $\cF$ a base-point-free linear series $g_n^r$, where $n=td-\deg B$.
 
\begin{prop}\label{eq2} $n \leq td-\sum\limits_{i=1}^k s_i r_i$, and the strict inequality holds if and only if some tangent line to $\cF$ at some $P_i$ is also a tangent line to every curve of $\Sigma_t$ at $P_i$. 
\end{prop}
\begin{proof}
The points in the support of $B$ are the points $P_{ij}$ on the non-singular model $\cX$ of $\cF$ corresponding to the branches $\gamma_{ij}$ of $\cF$ centered at $P_i$, $i=1\ldots,k$. Furthermore, if $\ord (\gamma)$ denotes the order of the branch $\gamma$, then
$$
v_{P_{ij}}(B)=\min\{I(P_i, \cC \cap \gamma_{ij}) \ | \ \cC \in \Sigma_t\} \geq s_i \cdot \ord(\gamma_{ij}),
$$
and the strict inequality holds if and only if the tangent line to $\gamma_{ij}$ is tangent to every curve of $\Sigma_t$ at $P_i$. Since $\sum\limits_{j} \ord(\gamma_{ij})=m_{P_i}(\cF)=r_i$, the result follows.
\end{proof}

From now on we assume that $g_n^r$ is simple, and that $P_1,\ldots,P_k$ are in general position (with respect to $t$ and $s_1,\ldots,s_k$). In particular,
\begin{equation}\label{dimensao}
r = h=\frac{1}{2}\,t(t+3)-\displaystyle\frac{1}{2} \sum_{i=1}^k s_i(s_i+1).
\end{equation}

Suppose that each $P_i$ is defined over $\fq$, for $i=1,\ldots,k$. Then $g_n^r$ is defined over $\fq$. By \cite[Theorem 2.13]{SV} and \cite[Theorem 4.4]{AB}, both bounds (\ref{sv}) and (\ref{ab}) associated $g_n^r$ can be written in terms of $d,t,r_i,s_i$. Denote by $\Gamma_t$ the linear system of all plane curves of degree $t$. Then $\Gamma_t$ cuts on $\cF$ a linear series $g_{n^\prime}^{r^\prime}$ of degree $n^{\prime}=td$ and dimension $r^{\prime}=\frac{1}{2}t(t+3)$ \cite[Section 7.7]{HKT}. Note that
$$
n^{\prime}-n=\sum_{i=1}^{k}t_i \geq \sum_{i=1}^{k}r_is_i \ \ \text{ and } \ \ r^{\prime}-r=\frac{1}{2} \sum_{i=1}^k s_i(s_i+1).
$$
Hence, if the $P_i$s are singular, and we impose $s_i < 2r_i-1$ for all $i=1,\ldots,k$, when we switch from $g_{n^\prime}^{r^\prime}$ to $g_n^r$, the degree decreases more than the dimension. This indicates that bounds (\ref{sv}) and (\ref{ab}) are potentially better when applied to $g_n^r$ than to $g_{n^\prime}^{r^\prime}$. The following example illustrates this fact concerning the St\"ohr-Voloch bound (\ref{sv}).

\begin{exa}
Consider the generalized Hurwitz curve $\cF$ of degree $d=n+l$ defined over $\fq$ by the affine equation $X^nY^l+X^l+Y^n=0$, where $1 \leq l \leq n$. The singularities of $\cF$ are $P_1=(0:0:1), P_2=(0:1:0)$ and $P_3=(1:0:0)$, each of them with multiplicity $l$. By Proposition (\ref{eq2}) and (\ref{dimensao}), the linear system of conics passing through $P_1$, $P_2$ and $P_3$ cuts out on $\cF$ a linear series $g_{2n-l}^2$. Thus, this gives rise to a plane model of $\cF$ of degree $2n-l$. Note that $2n-l<d$ provided that $2l>n$.  For instance, let us consider the case $n=4$ and $l=3$ over the field $\F_{17}$. In this case, the genus of $\cF$ is $6$. Let $g_{n+l}^{2}$ and $g_{2(n+l)}^{5}$ be the linear series obtained via the cut out on $\cF$ by the complete linear system of lines and conics respectively. Then in can be checked $\cF$ is $\F_{17}$-Frobenius classical with respect to these three linear series, and we have the following bounds:
\begin{equation}\label{tabela1}
\begin{tabular}{|c|l|c|}
\hline
Bound & $N_1(\cF) \leq$ \\ \hline 
St\"orh-Voloch bound via  $g_{2n-l}^{2}$& $52$ \\[1.5mm]
St\"orh-Voloch bound via  $g_{n+l}^{2}$&  $71$ \\[1.5mm]
St\"orh-Voloch bound via  $g_{2(n+l)}^{5}$ & $81$ \\[1.5mm]
Hasse-Weil (\ref{hw}) & $66$  \\ [1.5mm] \hline 
\end{tabular}
\end{equation}

\end{exa}

\section{Bounds for the number of branches of curves satisfying (H)}\label{twosing}

 Let us recall that $\cF$ is a plane curve of degree $d$ and genus $\gg$ defined over $\F_q$. In this section, we investigate the cases for which the hypothesis (H) presented in Section \ref{intro} holds, that is, 
\begin{itemize}
\item [(H)] $\cF$ has at least two singularities $P_1$ and $P_2$ defined over $\F_q$ with multiplicities $r_1$ and $r_2$, respectively, such that $r_1+r_2=d$.
\end{itemize}
Note that, in particular, $d \geq 4$. The main result of this section is the following.

\begin{theorem}\label{cotaSVH}
Let $\cF$ be a plane curve of degree $d$ defined over $\fq$ satisfying (H). Assume that $\cF$ is $\F_{q^m}$-Frobenius classical with respect to the linear system of conics passing through $P_1$ and $P_2$. Then
\begin{equation}\label{svc}
N_m(\cF) \leq 2r_1r_2+\frac{(r_1+r_2)(q^m-3)}{3}.
\end{equation}  
Furthermore, if $\cF$ is $(q,q^m)$-Frobenius classical, then there are integers $c_m \geq 2$, $c_{m-1} \geq q$ and $c_1 \geq q+4$ such that
\begin{equation}\label{abc}
N_m(\cF) \leq \frac{2r_1r_2+(r_1+r_2)q(q^{m-1}+1)-(c_1-c_m-c_{m-1})N_1(\cF)-c_{m-1}N_{m-1}(\cF)}{c_m}.
\end{equation}
\end{theorem}
\begin{proof}
Setting  $s_1=s_2=1$ in Proposition (\ref{eq2}) and (\ref{dimensao}), we conclude that the linear series $g_n^r$ cut out on $\cF$ by the linear system of conics passing through $P_1$ and $P_2$ is such that $n \leq d$ and $r=3$ \footnote{In general, $n=d$. For instance, this happens if $q>\frac{d-2}{2}$.}. From (\ref{eq4}) we obtain $\gg \leq r_1r_2-r_1-r_2+1$. Hence (\ref{svc}) follows from (\ref{sv}), and (\ref{abc}) follows from (\ref{ab}) and \cite[Proposition 4.1]{AB}. 
\end{proof}

\begin{cor}\label{corcota}
Let $\cF$ be a plane curve of degree $d$ defined over $\fq$ satisfying (H). Then both bounds (\ref{svc}) and (\ref{abc}) hold if one of the following is satisfied:
\begin{itemize}
\item [(a)] $\cF$ is $\F_{q^m}$-Frobenius classical with respect to the linear system of conics passing through $P_1$ and $P_2$ and $p>3$;
\item [(b)] $p>d$.
\end{itemize}
\end{cor}
\begin{proof}
Assume that $p>3$. If $\cF$ is $\F_{q^m}$-Frobenius classical with respect to $g_n^3$, then \cite[Corollary 3.9]{AB} implies that $\cF$ is $(q,q^m)$-Frobenius classical with respect to $g_n^3$, and the result follows. If $p>d$, it follows from \cite[Corollary 1.8]{SV} that $\cF$ is classical with respect to $g_n^3$. Thus \cite[Remark 8.52]{HKT} implies that $\cF$ is $\F_{q^m}$-Frobenius classical, and then we have the result from (a). 
\end{proof}

Let $g_d^2$ and $g_{2d}^5$ be the base-point-free linear series cut out on $\cF$ by the complete linear system of lines and conics, respectively. Then, if we assume that $\cF$ is $\F_{q^m}$-Frobenius classical with respect to these two linear series (this happens, in particular, if $p > 2d$, according to \cite[Corollary 1.8]{SV}), the St\"ohr-Voloch method (\ref{sv}) applied to  $g_d^2$ and $g_{2d}^5$ yields, respectively,

\begin{equation}\label{svf}
N_m(\cF) \leq r_1r_2+\frac{(r_1+r_2)q^m}{2}
\end{equation}  
and
\begin{equation}\label{svcl}
N_m(\cF) \leq 4r_1r_2+\frac{2(r_1+r_2)(q^m-5)}{5}.
\end{equation}  
Bound (\ref{svc}) is better than both bounds (\ref{svf}) and (\ref{svcl}) if
$$
q^m>\frac{6(r_1r_2-r_1-r_2)}{r_1+r_2}.
$$

\begin{exa}
Let $\cF$ defined over $\F_{13}$ by the affine equation $$X^6Y^6+X^6+Y^6-3=0.$$ It can be straightforwardly checked that $\cF$ satisfies (H). The genus of $\cF$ is $\gg=25$ and $N_1(\cF)=48$. According to Theorem \ref{cotaSVH}, Corollary \ref{corcota} and the above remarks, we have the following:
\begin{equation}\label{tabela1}
\begin{tabular}{|c|l|c|}
\hline
Bound & $N_2(\cF) \leq$ \\ \hline 
(\ref{svc})  & $736$ \\[1.5mm]
(\ref{abc})  &  $768$ \\[1.5mm]
Hasse-Weil (\ref{hw}) & $820$ \\[1.5mm]
(\ref{svcl}) & $931$  \\ [1.5mm] 
(\ref{svf}) & $1050$  \\ [1.5mm] \hline 
\end{tabular}
\end{equation}
\end{exa}

We point out that bound (\ref{abc}) improves bound (\ref{svc}) in many situations. For instance, if $m=2$, this happens in particular if 
\begin{equation}\label{N_1}
N_1(\cF) > \frac{q(q+3)(r_1+r_2)-6(r_1r_2-r_1-r_2)}{3(q+2)}.
\end{equation}
However, as will be seen in Section \ref{att}, bound (\ref{abc}) may improve bound (\ref{svc}) even if (\ref{N_1}) does not hold.

\section{Characterization of plane models satisfying (H) and a curve attaining bound (\ref{abc})}\label{att}

The aim of this section is twofold. First, we give a characterization, up to projective transformation, of the plane models of curves satisfying (H). Then, we finally present a curve for which equality holds in (\ref{abc}), namely the generalized Artin-Mumford curve. We point out that the results of Lemma \ref{pec}, Theorem \ref{carac} and Proposition \ref{toda} hold for curves over an arbitrary field $K$.

Set $\ell_1:X=0$, $\ell_2:Y=0$ and $\ell_\infty:Z=0$. Then, by \cite[Theorem 6.42]{HKT},
\begin{equation}\label{divpri}
\Div(x)=\ell_1 \cdot \cF-\ell_\infty \cdot \cF \ \ \ \ \textrm{ and } \ \ \ \ \Div(y)=\ell_2 \cdot \cF-\ell_\infty \cdot \cF.
\end{equation}

\begin{lemma}\label{pec}
Let $\cF:f(X,Y)=0$ be an irreducible plane curve of degree $d$ defined over a field $K$, and set $O_1=(1:0:0)$ and $O_2=(0:1:0)$. Let $x,y \in K(\cF)$ be such that $K(\cF)=K(x,y)$ and $f(x,y)=0$. If $x$ and $y$ have no common poles, then the following holds:
\begin{itemize}
\item [(i)] $\ell_\infty \cap \cF= \{O_1,O_2\}$.
\item [(ii)] $P \in \cX$ is a pole of $x$ if and only if the corresponding branch is centered at $O_1$, and $Q\in \cX$ is a pole of $y$ if and only if the corresponding branch is centered at $O_2$.  
\item [(iii)] $\ell_{\infty}$ is not tangent to $\cF$ neither at $O_1$ nor at $O_2$.
\end{itemize}
\end{lemma}
\begin{proof}
(i) and (ii) follow from (\ref{divpri}). For (iii), assume that $\ell_{\infty}$ is tangent to $\cF$ at $O_1$. From (\ref{muint}) and the fact that the number of tangents to $\cF$ at $O_1$ is finite, it follows that there exists a branch $\gamma$ of $\cF$ centered at $O_1$ for which $\ell_{\infty}$ is tangent to. Then $I(O_1,\ell_2 \cap \gamma) < I(O_1,\ell_\infty \cap \gamma)$. Let $P \in \cX$ be the point associated to $\gamma$. Since $P$ has weight $I(O_1,\ell_\infty \cap \gamma)$ on $\ell_\infty \cdot \cF$ and weight $I(O_1,\ell_2 \cap \gamma)$ on $\ell_2 \cdot \cF$, it follows that $P$ is a pole of $y$, contradicting (ii). Hence the assertion follows. The same argument holds for $O_2$. 

\end{proof}

\begin{theorem}\label{carac}
Let $\cF:f(X,Y)=0$ be an irreducible plane curve of degree $d$ defined over a field $K$. Let $x,y \in K(\cF)$ be such that $K(\cF)=K(x,y)$ and $f(x,y)=0$. The following assertions are equivalent:
\begin{itemize}
\item [(a)] The functions $x$ and $y$ have no common poles, with $\deg(\Div(x)_\infty)=l$ and $\deg(\Div(y)_\infty)=m$.
\item [(b)] $f(X,Y)=X^mY^l+g(X,Y)$, where $g(X,Y)$ is a polynomial of degree $n<d=m+l$, such that the degree of $g(T,Y) \in K[Y][T]$ is $\leq m$ and degree of $g(X,T) \in K[X][T]$ is $\leq l$.
\item [(c)] $O_1=(1:0:0)$ and $O_2=(0:1:0)$ have respective multiplicities $l$ and $m$ on $\cF$ such that $m+l=d$.
\end{itemize}
\end{theorem}
\begin{proof}
Let us suppose that (a) holds. Write $f(X,Y)=F_d+g(X,Y)$, where $F_d$ denotes the form of degree $d$ of $f(X,Y)$. From Lemma \ref{pec}(i) we conclude that $F_d=\alpha X^{\tilde{m}}Y^{\tilde{l}}$ for some $\alpha \in K^*$, with ${\tilde{m}},{\tilde{l}} >0$. We may assume that $\alpha=1$, and so $\cF:X^{\tilde{m}}Y^{\tilde{l}}+g(X,Y)=0$, where $g(X,Y) \in K[X,Y]$ has degree $n<d$. We will proceed to prove that the variable $Y$ has degree $ \leq {\tilde{l}}$ on $g(X,Y)$, with the proof for the variable $X$ being analogous.  

Let $a_i(x) \in K[x]$, with $i \in \{0,\ldots,k\}$ and $a_k(x) \neq 0$, be such that
$$
g(X,Y)=a_0(X)+a_1(X)Y+\cdots+a_k(X)Y^k.
$$
Set $d_i:=\deg(a_i(X))$. If $F(X,Y,Z)$ denotes the homogenization of $F(X,Y)$ with respect to the variable $Z$, then $$F(X,Y,Z)=X^{\tilde{m}}Y^{\tilde{l}}+G(X,Y,Z)Z^{{\tilde{m}}+{\tilde{l}}-n},$$ where
$$
G(X,Y,Z)=A_0(X,Z)Z^{n-d_0}+A_1(X,Z)YZ^{n-d_1-1}+\cdots+A_k(X,Z)Y^kZ^{n-d_k-k},
$$
with $A_i(X,Z)$ denoting the homogenization of $a_i(X)$ with respect to $Z$. Thus  
\begin{equation}\label{deso}
F(X,1,Z)=X^{\tilde{m}}+A_0(X,Z)Z^{{\tilde{m}}+{\tilde{l}}-d_0}+A_1(X,Z)Z^{{\tilde{m}}+{\tilde{l}}-d_1-1}+\cdots+A_k(X,Z)Z^{{\tilde{m}}+{\tilde{l}}-d_k-k}.
\end{equation}
Note that $d_k+k \leq n <{\tilde{m}}+{\tilde{l}}$. From (\ref{deso}) we have $m_{O_2}(\cF)=\min\{{\tilde{m}},{\tilde{m}}+{\tilde{l}}-k\}$. Assume that $k>{\tilde{l}}$. Then the tangents to $\cF$ at $O_2$ are defined by the linear factors of $A_k(X,Z)Z^{{\tilde{m}}+{\tilde{l}}-d_k-k}$. In particular, $\ell_\infty$ is one of these tangents, contradicting Lemma \ref{pec}(iii). Therefore ${\tilde{l}} \geq k$. It remains to show that ${\tilde{m}}=m$ and ${\tilde{l}}=l$. But this follows from 
\begin{equation}\label{polediv}
\deg(\Div(x)_\infty) =[K(\cF):K(x)] \ \ \ \text{and} \ \ \ \deg(\Div(y)_\infty)= [K(\cF):K(y)], 
\end{equation}
where $[K(\cF):L]$ denotes the degree of the field extension $K(\cF)/L$ (see e.g. \cite[Theorem 1.4.11]{St}).

Conversely, assume that (b) holds. It is easy to see that $\ell_\infty \cap \cF= \{O_1,O_2\}$. With $F(X,Y,Z)$ denoting the homogenization of $f(X,Y)$ w.r.t. Z, write $F(X,1,Z)$ as in (\ref{deso}). Thus $I(O_2,\ell_\infty \cap \gamma)\leq I(O_2,\ell_1 \cap \gamma)$ for all branches $\gamma$ of $\cF$ centered at $O_2$. In the same way, one can see that $I(O_1,\ell_\infty \cap \xi)\leq I(O_1,\ell_2 \cap \xi)$ for all branches $\xi$ of $\cF$ centered at $O_1$. Hence it follows from (\ref{divpri}) that $x$ and $y$ have no common poles. By (\ref{polediv}),  equivalence between (a) and (b) is established. 

Now, (b) implies (c) straightforwardly. Assume that (c) holds. Since the multiplicity of $O_2$ in $\cF$ is $m$, we have
\begin{equation}\label{homo}
F(X,Y,Z)=F_m(X,Z)Y^l+F_{m+1}(X,Z)Y^{l-1}+\cdots+F_d(X,Z),
\end{equation}
where $F_i(X,Z)$ is homogeneous of degree $i$. Thus
\begin{equation}\label{efe}
f(X,Y)=F_m(X,1)Y^l+F_{m+1}(X,1)Y^{l-1}+\cdots+F_d(X,1).
\end{equation}
In particular, we see that $Y$ has degree $l$ on $f(X,Y)$; analogously, it can be shown that $X$ has degree $m$ on $f(X,Y)$. Hence, $X$ has degree $\leq m$ on each term of the sum on the right side of (\ref{efe}). Therefore, (b) holds.
\end{proof}

The next result shows that hypothesis (H) is not so restrictive. 

\begin{prop}\label{toda}
Let $\cF$ be a plane curve defined over a field $K$ with at least two branches defined over $K$. Then $\cF$ has a plane model satisfying (H). 
\end{prop}
\begin{proof}
Denote by $d$ the degree of $\cF$ and consider the $K$-rational branches $\gamma_1$ and $\gamma_2$ of $\cF$ centered at $P_1$ and $P_2$, respectively.  Denote by $r_1$ and $r_2$ the multiplicities of $P_1$ and $P_2$, respectively. After a $K$-projective transformation, we may assume that $P_1=(1:0:0), P_2=(0:1:0)$, and $(0:0:1) \notin \cF$. Then, it follows from \cite[Section 7.4 (2) and (3)]{Fu} that the Cremona standard quadratic transformation $(X:Y:Z)\mapsto (YZ:XZ:XY)$ gives rise to a plane model of degree $2d-r_1-r_2$ for which $(1:0:0)$ and $(0:1:0)$ have multiplicities $d-r_1$ and $d-r_2$. Therefore, the result follows from Theorem \ref{carac}.
\end{proof}

\subsection{The generalized Artin-Mumford curve}

Henceforth, the characteristic of the field $\F_q$ is assumed to be $p>3$. The generalized Artin-Mumford curve $\cM$ is the projective closure of the curve defined by the affine equation
\begin{equation}\label{gam}
\cM:(X^q-X)(Y^q-Y)=1
\end{equation}
over the ground field $\F_q$. It is straightforward to check that $\cM$ has only two singularities, namely $O_1=(1:0:0)$ and $O_2=(0:1:0)$, both $q$-fold ordinary. In particular, $\cM$ satisfies (H), and from (\ref{eq4}) it follows that it has genus $\gg=(q-1)^2$.

Let $\F_q(x,y)$ be the function field of $\cM$, where $(x^q-x)(y^q-y)=1$. Let us consider the linear system $\Sigma$ of conics passing through $O_1$ and $O_2$. Denote by $g_{n}^3$ the linear series arising from $\Sigma$. Note that the conic $\cC \in \Sigma$, where $\cC$ is defined by $Z^2=0$, does not share tangents with $\cM$ neither at $O_1$ nor at $O_2$. Thus  from Proposition (\ref{eq2}), we have that $n=2q$. Furthermore, $\{1,x,y,xy\}$ is a basis for the Riemann-Roch space associated to $g_{2q}^3$.

Recall that $N_r(\cM)$ denotes the number of $\mathbb{F}_{q^r}$-rational points on a non-singular model of $\cM$. The bound (\ref{ab}) with $m=2$ and $u=1$ obtained via $\Sigma$ reads
\begin{equation}\label{bam}
c_1N_1(\cM)+c_2(N_2(\cM)-N_1(\cM)) \leq 2q\big((\kappa_0+\kappa_1)(q-2)+q^2+q+2 \big), 
\end{equation}
where $(\kappa_0,\kappa_1)$ denotes the $(q,q^2)$-Frobenius order sequence of $\cM$ with respect to $\Sigma$. If $(\kappa_0,\kappa_1)=(0,1)$, then (\ref{bam}) coincides with (\ref{abc}). In what follows, we show that equality holds in (\ref{bam}). As before, $\cX$ stands for a fixed non-singular model of $\cM$. We start by counting the rational points on $\cX$.

\begin{prop}\label{nop}
$N_1(\cM)=2q$ and $N_2(\cM)=q^2(q-1)+2q$.
\end{prop}
\begin{proof}
Clearly, $(X^q-X)(Y^q-Y)=1$ has no $\F_q$-rational solutions. The tangent lines to $\cM$ at $O_1$ are defined by $y=a$, with $a \in \F_q$, and the tangent lines to $\cM$ at $O_2$ are defined by $x=b$ with $b \in \F_q$. Hence $N_1(\cM)=2q$.

In order to compute $N_2(\cM)$, we only have to count the number of $\F_{q^2}$-rational solutions of $(X^q-X)(Y^q-Y)=1$, since those are all non-singular. To this end, consider the map
$$
\begin{array}{cccc}
\psi:&\F_{q^2} & \longrightarrow & \F_{q^2}\\
     &\alpha   & \longmapsto     & \alpha^q-\alpha.
\end{array}
$$

We claim that $\delta^{-1} \in \Im(\psi)$ for all nonzero $\delta \in \Im(\psi)$. Indeed, since $\Im(\psi)$ is the kernel of the trace map of $\F_{q^2}$ onto $\F_q$, we have $\delta \in \Im(\psi) \Leftrightarrow \delta^q+\delta=0$. Thus
$$
0=\frac{\delta^q+\delta}{\delta^{q+1}}=\delta^{-1}+\delta^{-q},
$$
whence the claim holds. 
For $a \in \F_{q^2}^{*}$, let $b \in \F_{q^2}^{*}$ such that $\psi(b)=\psi(a)^{-1}$. Then $\psi(a)\psi(b)=1$, which means that $(a,b)$ is an affine $\F_{q^2}$-rational point of $\cM$. Taking into account the fibers of $\psi(a)$ and $\psi(a)^{-1}$ via $\psi$, we obtain $q^2$ distinct points associated to $\psi(a)$. Thus from $\#(\Im(\psi)^{*})=q-1$, we have $q^2(q-1)$ $\F_{q^2}$-rational points obtained in this way. Since every affine $\F_{q^2}$-rational point $P \in \cM$ arises from the procedure described above, we conclude that the number of $\F_{q^2}$-rational points of $\cM$ that are not $\fq$-rational is $q^2(q-1)$. Therefore, the result follows by adding the $2q$ points that are $\F_{q}$-rational.
\end{proof}

The aim of the following sequence of results is to compute $\kappa_0$ and $\kappa_1$, and to estimate $c_1$ and $c_2$ in bound (\ref{bam}). Note that if $abc \neq 0$ the point $P=(a:b:c) \in \cM$ is non-singular; so it can be identified with an unique point of $\cX$.

\begin{lemma}\label{oscon}
Let $P=(a:b:1) \in \cM$. The osculating conic to $\cM$ at $P$ is the projective closure $\cH_P$ of the irreducible conic defined by $g_P(X,Y)=0$, where
\begin{equation}\label{eqosc}
g_P(X,Y)=(ab-1)^q-a^qY-b^qX+XY.
\end{equation} 
In particular, $\cH_P \in \Sigma$ for all $P=(a:b:1) \in \cM$. Furthermore, $I(P,\cM \cap \cH_P)\geq q$ for all such $P$.
\end{lemma}
\begin{proof}
Set $f(x,y)=(x^q-x)(y^q-y)-1$. Equation $f(x,y)=0$ can be written as
\begin{equation}\label{auxam}
(xy-1)^q-x^qy-y^qx+xy=0.
\end{equation}
Denote by $\cC_P$ the projective closure of the curve defined by $g_P(X,Y)=0$. From (\ref{auxam}) we obtain
$$
g_P(x,y)=g_P(x,y)-f(x,y)=(ab-xy)^q-(b-y)^qx-(a-x)^qy \in \bfq(\cM).
$$
Hence $v_P(g_P(x,y)) \geq q$, i.e., $I(P,\cM \cap \cC_P)\geq q$. It is easy to check that $\cC_P$ is non-singular; in particular, it is irreducible. Let $\cH_P$ be the osculating conic to $\cM$ at $P$. Then \cite[Theorem 1]{GV1} implies that $I(P,\cM \cap \cH_P) \geq q$. Thus by \cite[Lemma 3.3]{AB2} we have $I(P,\cC_P \cap \cH_P) \geq q$. As we are assuming that $p>3$,
$$
I(P,\cC_P \cap \cH_P) \geq q>4=\deg(\cC_P)\cdot\deg(\cH_P).
$$
Therefore, by B\'ezout's Theorem, the curves $\cC_P$ and $\cH_P$ have a common component. However, since $\cC_P$ is irreducible, it follows that $\cC_P=\cH_P$.

\end{proof}

\begin{lemma}\label{os}
The curve $\cM$ is non-classical w.r.t. $\Sigma$ with order sequence $(0,1,2,q)$.
\end{lemma}
\begin{proof}
Let $(\varepsilon_0,\varepsilon_1,\varepsilon_2,\varepsilon_3)$ be the order sequence of $\cM$ w.r.t. $\Sigma$. Here $\varepsilon_0=0$, $\varepsilon_1=1$ and $\varepsilon_2 \geq 2$. In view of Lemma \ref{oscon}, $\varepsilon_3 \geq q$. Assume that $\varepsilon_2 >2$. By \cite[Corollary 1.9]{SV}, it follows that $\varepsilon_2 \geq p$. Let $P \in \cX$ be a $\Sigma$-ordinary point, that is, $j_i(P)=\varepsilon_i$ for all $i=0,1,2,3$. Choose a conic $\tilde{\cC}_P$ for which $I(P,\cM \cap \tilde{\cC}_P)=\varepsilon_2$ and let $\cH_P$ be the osculating conic to $\cM$ at $P$. Note that $\cH_P \neq \tilde{\cC}_P$. Since $\varepsilon_2 \geq p$, from  \cite[Lemma 3.3]{AB2} we have
$$
I(P,\tilde{\cC}_P \cap \cH_P) \geq p \geq 5> \deg(\tilde{\cC}_P) \cdot \deg(\cH_P).
$$
Again by B\'ezout's Theorem, the curves $\tilde{\cC}_P$ and $\cH_P$ have a common component. But from Lemma \ref{oscon}, the conic $\cH_P$ is irreducible. Thus $\cH_P=\tilde{\cC}_P$, a contradiction. Hence $\varepsilon_2=2$. In view of (\ref{auxam}) and \cite[Theorem 7.65]{HKT}, we conclude that $\varepsilon_3=q$.
\end{proof}

Let $\cM(\F_q)=\{P_a, Q_b \ | \ a,b \in \F_q\}$ be the set of $\F_q$-rational points of $\cX$. Here, each $P_a$ (resp. $Q_b$) corresponds to a branch of $\cM$ centered at $O_2$ (resp. $O_1$) for which the line $x=a$ (resp. $y=b$) is tangent to. Regarding $\F_q(x,y)$ as an elementary $p$-abelian extension of degree $q$ of $\F_q(x)$, we obtain that $\div(x-a)_0=qP_a$ for all $a \in \F_q$; see \cite[Proposition 3.7.10]{St}. 
Lemma \ref{pec}(ii) gives us $\div(x-a)_\infty=\sum_{b \in \F_q}Q_b$. By the symmetry of the polynomial $f(X,Y)=(X^q-X)(Y^q-Y)-1$, we obtain for all $a,b \in \F_q$ that
\begin{equation}\label{divisor}
\div(x-a)=qP_a-\displaystyle\sum_{b \in \F_q}Q_b \ \ \ \text{ and } \ \ \  \div(y-b)=qQ_b-\displaystyle\sum_{a \in \F_q}P_a.
\end{equation}

\begin{lemma}\label{pos}
Let $\gamma$ be an $\fq$-rational branch of $\cM$. The order sequence of $\gamma$ with respect to $\Sigma$ is $(0,1,q,q+1)$.
\end{lemma}
\begin{proof}
Let $P \in \cX$ be the point corresponding to $\gamma$. Without loss of generality, we may assume that $P=P_a$ for some $a \in \F_q$. Then (\ref{divisor}) implies $v_P(y^{-1})=1$ and $v_P(x-a)=q$. Furthermore, we have $v_P((x-a)y^{-1})=q+1$. Since $\{1,y^{-1},x-a,(x-a)y^{-1}\}$ is another basis of the Riemann-Roch space associated to $g_{2q}^3$, the proof is complete.
\end{proof}

\begin{prop}\label{fnc}
Let $r$ be a positive integer. The curve $\cM$ is $\F_{q^r}$-Frobenius non-classical w.r.t. $\Sigma$ if and only if $r=2$.
\end{prop}
\begin{proof}
First, note that $x$ is a separating variable of $\fq(\cM)$. From (\ref{fncl}), the curve $\cM$ is $\F_{q^r}$-Frobenius non-classical if and only if
\begin{equation}\label{wrons}
\left|
\begin{array}{cccc}
1 & x^{q^r} &y^{q^r} & x^{q^r}y^{q^r} \\
1 & x       &y       & xy \\
0 & 1       & D_x^{(1)}(y) & D_x^{(1)}(xy)\\
0 & 1       & D_x^{(2)}(y) & D_x^{(2)}(xy) 
\end{array} \right|=0.
\end{equation} 
 Since the order sequence of $\cM$ is $(0,1,2,q)$ (Lemma \ref{os}), we conclude from (\ref{wrons})  and \cite[Corollary 1.3]{SV} that $\cM$ is $\F_{q^r}$-Frobenius non-classical w.r.t. $\Sigma$ if and only if $\Phi_{q^r}(P) \in \cH_P$ for infinitely many points $P \in \cX$, where $\cH_P$ is the osculating conic to $\cM$ at $P$ and $\Phi_{q^r}:\cX \rightarrow \cX$ is the $\F_{q^r}$-Frobenius map. Thus by Lemma \ref{oscon}, the last assertion is equivalent to the function
\begin{equation}\nonumber 
(xy-1)^q-x^qy^{q^r}-y^qx^{q^r}+x^{q^r}y^{q^r}=\big((x^{q^{r-1}}-x)(y^{q^{r-1}}-y)-1 \big)^q
\end{equation}
being vanishing. Therefore, $\cM$ is $\F_{q^r}$-Frobenius non-classical with respect to $\Sigma$ if and only if
\begin{equation}\label{eqdiv}
(X^q-X)(Y^q-Y)-1 \textrm{ divides } (X^{q^{r-1}}-X)(Y^{q^{r-1}}-Y)-1.
\end{equation}
Since the polynomial $(X^{p^m}-X)(Y^{p^m}-Y)-1$ is absolutely irreducible for all $m \geq 0$, assertion (\ref{eqdiv}) is only possible if $r=2$. Hence, the result follows.
\end{proof}

\begin{cor}\label{dfam}
The curve $\cM$ is $(q^u,q^m)$-Frobenius classical w.r.t. $\Sigma$ for all integers $u,m>0$ such that $u \neq m$ and $\gcd(u,m)=1$.
\end{cor}
\begin{proof}
Suppose that $\cM$ is $(q^u,q^m)$-Frobenius non-classical w.r.t. $\Sigma$ for positive integers $u,m$, with $u \neq m$ and $\gcd(u,m)=1$. Let $(\kappa_0,\kappa_1)$ be the $(q^u,q^m)$-Frobenius order sequence of $\cM$. Since $N_1(\cM)>0$, it follows from \cite[Proposition 3.4]{AB} that $\kappa_0=0$. Then \cite[Corollary 3.9]{AB} implies that $\cM$ is $\F_{q^r}$-Frobenius non-classical w.r.t. $\Sigma$ for $r=u$ and $r=m$, contradicting Proposition \ref{fnc}.
\end{proof}

We are now able to show that equality holds in (\ref{bam}). Let $\gamma$ be an $\fq$-rational branch of $\cM$ with order sequence $(j_0,j_1,j_2,j_3)$ w.r.t. $\Sigma$. According to \cite[Proposition 4.1]{AB},
$$
c_1 \geq qj_1+j_2-\kappa_0+j_3-\kappa_1.
$$
From Lemma \ref{pos}, we have $(j_0,j_1,j_2,j_3)=(0,1,q,q+1)$. Furthermore, $(\kappa_0,\kappa_1)=(0,1)$ by Corollary \ref{dfam}. Thus $c_1 \geq 3q$. From \cite[Proposition 4.1]{AB}, $c_2 \geq 2$. Therefore, the conclusion follows.

\subsection*{Acknowledgments}
This research was supported by FAPESP-Brazil, grant 2013/00564-1. The author would like to thank Herivelto Borges and G\'abor Korchm\'aros for many useful discussions on the topic of this article.

\vspace{0,5cm}\noindent {\em Author's address}:

\vspace{0.2 cm} \noindent Nazar ARAKELIAN \\
Centro de Matem\'atica, Computa\c c\~ao e Cogni\c c\~ao
\\ Universidade Federal do ABC \\ Avenida dos Estados, 5001 \\
CEP 09210-580, Santo Andr\'e SP
(Brazil).\\
 E--mail: {\tt n.arakelian@ufabc.edu.br}

\end{document}